%% file: main.tex
\documentclass[a4paper,11pt]{article}

\include{header.tex}

\begin{document}


\title{The homotopy type of the moment-angle complex associated to the complex of injective words}

\author{
Pedro Conceição
}


\date{}

\maketitle
\begin{abstract}
Topological methods have emerged as valuable tools for analyzing the structural properties of directed graphs, particularly connectome data in computational neuroscience. This paper investigates the construction of topological spaces from combinatorial data of directed graphs using the polyhedral product functor, with particular emphasis on understanding their homotopy type, which is also of independent interest in topology and combinatorics. Specifically, we compute the homotopy type of the moment-angle complex over the face poset of the complex of injective words. This reveals a tight connection between homotopy and combinatorics: its homotopy type is determined by the $h$-vector of complexes of injective words. We also construct an associated homotopy fibration of polyhedral products associated to ordered simplicial complexes, which in a way generalizes the analogous homotopy fibration for polyhedral products over abstract simplicial complexes.
\end{abstract}
\vspace{1cm}

\section*{Introduction}
Ideas from combinatorial topology have recently found applications in computational neuroscience. A successful example is the construction of a topological space from the combinatorial data of an underlying connectome to study its global structure. In this context, a connectome is represented as a directed graph, commonly referred to as a digraph. Topological spaces serve as powerful tools for investigating connectome architecture, as their inherent properties reveal insights into the network's structural organization. In particular, the directed flag complex is a higher dimensional topological space made out of directed cliques of the underlying directed graph, which has recently been used as such a tool and provided valuable insights to functional properties of the connectome \cite{neidigraphs_network,bluebrain_paper}. 

This paper investigates polyhedral products as a  more general construction of topological spaces from combinatorial data of digraphs and directed flag complexes, focusing on understanding their homotopy type. polyhedral products are topological spaces contructed by gluing together other spaces according to a combinatorial recipe determined by a poset $\PP$. Introduced by Bahri, Bendersky, Cohen, and Gitler \cite{BBCG}, the polyhedral product space generalizes two fundamental families of topological spaces in toric topology: Davis-Januszkiewicz spaces, denoted by $DJ_\PP$ or $\ZZ_\PP(\C P^\infty,*)$, and moment-angle complexes, denoted by $\ZZ_\PP$ or $\ZZ_\PP(D^2,S^1)$ \cite{davis-januszkiewicz-seminal}. In particular, polyhedral products are constructed from diagrams of spaces. Connection can be drawn to the work of Ziegler and \v{Z}ivaljevi\'{c} \cite{homo-type-diagram}, developed further by Welker, Ziegler and \v{Z}ivaljevi\'{c} \cite{hocolim-comb}, which are the first combinatorial application of the theory of diagram of spaces.

While polyhedral products were originally defined specifically for cases where $\PP$ is the face poset of a simplicial complex, this concept has been extended to increasingly general poset structures \cite{lue-panov,kishimoto2019polyhedral}. This generalization is essential for our study, as a directed flag complex is generally not a simplicial complex. The simplest example consists of a topological space defined by two vertices reciprocally connected; this a directed flag complex but it is not a simplicial complex. 

Specifically, we study the moment-angle complex associated with the complex of injective words on $[n]:=\{1,…,n\}$ letters, which we denote by $\ZZ_{\Gamma_n}$. The complex of injective words is a well-studied object in mathematics \cite{farmer1978cellular,lex-wachs-bjoern,injective-words-ran-meshulam-wojtek,injective-words-stability} and, in our setting, can be defined as the directed flag complex of the complete $n$-digraph. Directed flag complex and the complex of injective words are examples of what we call ordered simplicial complexes. An ordered simplicial complex on $[n]$ is a collection of finite ordered subsets of $[n]$ that is closed under taking ordered subsets. Its 1-skeleton is a digraph with no loops and no multiple edges in the same direction.

While the cohomology of moment-angle complexes is well understood, their homotopy type remains largely unexplored. Previous work has identified cases where such complexes are homotopy equivalent to wedges of spheres \cite{shifted-GRBIC20136,GRBIC2007357_coordinatespace,shifted-IRIYE2013716}; however, these results have been restricted to moment-angle complexes associated to simplicial complexes. Our work extends this investigation to the ordered simplicial complex setting. 

Our first main results starts from observation that the additive structure of the cohomology of the moment angle complex associated to the complex of injective words $\ZZ_{\Gamma_n}$ can be precisely computed using the combinatorics properties of the complex of injective words. Based on this, we show that the homotopy colimit construction of $\ZZ_{\Gamma_n}$ can be computed precisely, yielding a tight connection to combinatorial proprieties of the complex on injective words.

\begin{alphabeticaltheorem}\label{theomA}
\[
\ZZ_{\Gamma_n} \simeq \bigvee_k h_k S^{2k} \text{ with } n\geq 2 \text{ and }  2 \leq k \leq n,
\]
where $h_k$ is the $k$th element of the $h$-vector of the complex of injective words and $h_k S^{2k}$ denotes a wedge sum of $h_k$ spheres of dimension $2k$.
\end{alphabeticaltheorem}

Our second result starts from the observation that any ordered simplicial complex injects into the
complex of injective words on the same vertex set. From this observation and Theorem \ref{theomA}, we obtain the following. 

\begin{alphabeticaltheorem}\label{theomB}
There is a homotopy fibration given by 
\[
\ZZ_\PP(\Omega(\bigvee h_k S^{2k})\times D^2, \Omega(\bigvee h_k S^{2k})\times S^1 )	\to \ZZ_\PP(\C P^\infty,*)\to \ZZ_{\Gamma_n}(\C P^\infty,*),
\]
where $\PP$ and $\Gamma_n$ are, respectively, the face poset of an ordered simplicial complex and of the complex of injective words both on $n$ vertices, and $h_k$ is the the $k$th element of the $h$-vector of the complex of injective words.
\end{alphabeticaltheorem}

Theorem \ref{theomB} generalizes the well-known homotopy fibration of polyhedral products, which, like the ordered case, relies on the fact that any simplicial complex injects into the simplex on the same vertex set.
\[
\ZZ_{K}(D^2,S^1) \to \ZZ_K(\C P^\infty, *) \to \ZZ_{\Delta^{(n-1)}}(\C P^\infty, *) = \prod_{i = 1} ^n\C P^\infty,
\]
where $K$ is a simplicial complex on $n$ vertices and $\Delta^{(n-1)}$ is the $n-1$-simplex.

\paragraph{Organization.} We organize this paper as follows. Section \ref{sec-preliminaries} contains  preliminary material, establishes our notations and recall some known results. Section \ref{sec-homoto-type} is dedicated to the proof of Theorem \ref{theomA} (as Theorem \ref{homo-type-inj-moment-angle}). Finally, we prove Theorem \ref{theomB} (as Theorem \ref{mark-suggestion-new-fibration}) in Section \ref{sec-fibration}.

\section{The main objects and their preliminaries}\label{sec-preliminaries}

\paragraph{Ordered simplicial complexes and simplicial posets.} An \emph{ordered simplicial complex} $K$ on $[n]$ is a collection of finite ordered subsets of $[n]$ (also called simplices) that is closed under taking ordered subsets. We always assume that $\emptyset$ belongs to $K$. Our motivation for studying ordered simplicial complexes stems from their construction from digraphs (without loops or multiple edges in the same direction), which can model biological connectomes. A key example is the \emph{directed flag complex}, an ordered simplicial complex constructed from a digraph 
$G$ on $[n]$ vertices. The vertices and {$1$-skeleton} of the directed flag complex are determined by $G$, and there is a $k$-cell for each $(k+1)$-directed clique in $G$. A $(k+1)$-directed clique  is an ordered set of vertices $(v_0, \ldots, v_k)$ such that there is an edge from $v_i$ to $v_j$ in $G$ whenever $0 \leq i < j \leq k$. 

A central object of our study is the \emph{complex of injective words}, a directed flag complex of the complete digraph whose simplices are the injective words viewed as directed cliques in the underlying digraph. Recall that the complete $n$-digraph is a directed graph on $[n]$ vertices in which every pair of distinct vertices is connected by a pair of unique edges (one in each direction). 

\begin{definition} 
    The \textbf{complex of injective words} $\inj[n]$ is a directed flag complex whose $1$-skeleton is the complete $n$-digraph on $[n]$ vertices. Its $k$-simplices are indexed by $(k+1)$-ordered tuples $(v_0, \ldots, v_k)$ of distinct vertices in $[n]$, and its face inclusions correspond to the containment reltation: $(v_{i_1}, \ldots, v_{i_l})$ is a face of $(v_0, \ldots, v_k)$ when $\{v_{i_1}, \ldots, v_{i_l}\} \subseteq \{v_0, \ldots, v_k\}$ preserving the induced order.
\end{definition}

The face poset of an ordered simplicial complex $K$ is the poset $\PP$ whose set of elements are the simplices of $K$ and whose partial order relation is the inclusion relation on the set of simplices, that is, $\sigma \subseteq \tau \in K$ implies $\sigma \leq \tau \in \PP$. Let $\Gamma_n$ denote the face poset of $\inj[n]$. We shall often write $\PP$ for both the complex and its face poset. In particular, we shall often write $\Gamma_n$ for the complex of injective words and its face poset. Face posets of ordered simplicial complexes are examples of the so-called \emph{simplicial posets}. 

\begin{definition}
A poset $\PP$ with order relation $\leq$ is called \textbf{simplicial} if it has an initial element $\emptyset$ and for each $\sigma \in \PP$ the lower segment 
\[
\PP_{\leq \sigma} := \{\tau \in \PP: \emptyset \leq \tau \leq \sigma\}
\]
is the face poset of a simplex.
\end{definition}

Let $V_\PP$ denote the vertex set of $\PP$, consisting  of elements $v \in \PP$ such that $p \leq v$ if and only if $p = \emptyset \in \PP$. When $|V_\PP| =n $, we say that $\PP$ is based on $[n]$ vertices. For each $\sigma \in \PP$, let $V_\PP(\sigma)$ be the set of vertices of $\sigma$, that is,  $v \in V_\PP$ such that $v \leq \sigma$. Note that if $\PP$ is a simplicial poset and $\sigma \leq \tau \in \PP$ with $V_\PP(\sigma) = V_\PP(\tau)$, then $\sigma=\tau$. For $a \subseteq [n]$, let $\PP_a$ denote the subposet of $\PP$, consisting of elements $\sigma \in \PP$ with $V(\sigma) \subseteq a$.

Simplicial posets generalize naturally abstract simplicial complexes as face posets of the so called \emph{simplicial cell complexes}, also known as semi-simplicial complexes or $\Delta$-complexes. We adopt the terminology of \cite[Chapter~2.8]{buchstaber2014toric} and refer to them as simplicial cell complexes. 

Briefly, a simplicial cell complex is almost a simplicial complex, but has more flexible rules for attaching simplices along common faces - for example, simplices can share the same vertex set. In particular, an ordered simplicial complex is an example of a simplicial cell complex. The simplest example of a simplicial poset that is not a face poset of a simplicial complex is $\Gamma_2$, that is:
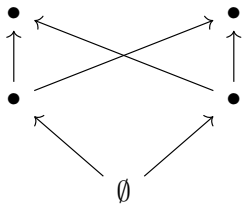
\[\begin{tikzcd}
    12 && 21 \\
    1 && 2 && \\
    & \emptyset
    \arrow[from=2-1, to=1-1]
    \arrow[from=2-1, to=1-3]
    \arrow[from=2-3, to=1-1]
    \arrow[from=2-3, to=1-3]
    \arrow[from=3-2, to=2-1]
    \arrow[from=3-2, to=2-3]
\end{tikzcd}\] 
\quad the face poset of  
\begin{tikzcd}\bullet_1 \arrow[r, shift left, bend left=10, no head, "12"] & \bullet_2 \arrow[l, shift left, bend left=10, no head, "21"]\end{tikzcd}.

We restrict our attention to finite simplicial posets; unless stated otherwise, $\PP$ denotes either such a poset or its corresponding simplicial cell complex, with the distinction clear from context. Although polyhedral products have been recently generalized to more general posets \cite{kishimoto2019polyhedral} - which include simplicial posets as a special case \cite[Example~4.2]{kishimoto2019polyhedral} - we work exclusively with finite simplicial posets for our purposes. While some of our results may be extended to more general settings, doing so would introduce unnecessary notational complexity without substantial benefit to this paper. 

\paragraph{Face vectors of a simplicial cell complex.} One can associate vectors to simplicial cell complexes vectors that encode the number of faces of different dimensions, namely the $f$-vector and $h$-vector. 
\begin{definition}
\label{dfn-f-h-vector}
	Let $\PP$ be a simplicial cell complex of dimension $n$. 
	\begin{itemize}
		\item The $f$-vector is given by $\textbf{f}(\PP) = (f_0, \ldots, f_{n-1})$, 
		where $f_i$ is  the number of $i$-dimensional faces of $\PP$. 
		
		\item The $h$-vector $\textbf{h}(\PP) = (h_0, \ldots, h_{n})$ is defined by the identity 
		\begin{align*}
			h_k &= \sum_{i=0}^{k} (-1)^{k-i} {n-i \choose k-i} f_{i-1} 
		\end{align*}
		with  $f_{-1} := 1$ and $0 \leq k \leq n$.
	\end{itemize}
\end{definition}

The $f$- and $h$-vectors are well studied combinatorial invariants. Stanley provided a complete characterization of $f$-vectors for simplicial posets \cite{STANLEY-fhvectors}, and they satisfy interesting symmetric relations, like the Dehn-Sommerville Relations \cite[Chapter~1.3]{buchstaber2014toric}.
    
\paragraph{Properties of the complex of injective words.} We now recall important properties of the complex of injective words. 

\begin{definition}
 	A simplicial cell complex is pure if its facets (maximal dimension faces) have the same dimension.
\end{definition}

\begin{definition} A pure simplicial cell complex $K$ of dimension $d$ is \textbf{shellable} if one can linearly order its facets $F_1, \ldots, F_k$ so that for all $j \geq 2$, $F_j \cap (F_1 \cup F_2 \cup \ldots \cup F_{j-1})$ is a pure $(d-1)$-dimensional subcomplex of $\partial F_j$, the boundary of $F_j$. Such an order is called shelling order.
 \end{definition} 

Although shellability can be defined for more general complexes \cite{bj-wa-non-pure-shell}, the above definition suffices for our purposes. One way to think of shelling order is the following: the simplicial cell complex is built attaching facet by facet in a particular sequence where $F_j$ is attached to $(F_1 \cup F_2 \cup \ldots \cup F_{j-1})$ over $F_j \cap (F_1 \cup F_2 \cup \ldots \cup F_{j-1})$. The combinatorial condition of being shellable has strong topological consequences. 
 
\begin{theorem}\cite[Proposition~4.3]{bj-shell} 
  Let $\Delta$ be a pure simplicial cell complex of dimension $d+1$. If $\Delta$ is shellable, then it is homotopy equivalent to a wegde sum of $d$-spheres and the number of $d$-spheres is indexed by the $d$-dimensional facets that are attached over the entire boundary.
\end{theorem}
\begin{theorem}\cite[Theorem~6.1]{lex-wachs-bjoern}\label{theo-shell-inj}
 The complex of injective words $\inj[n]$ is shellable via the lexicographic order $\leq_l$ of the symmetric group $\Sigma_n$. Moreover, $\inj[n] \simeq \bigvee_{d(n)} S^{n-1}$, where $d(n)$ is the number of derangements in $\Sigma_n$, that is, the number of fixed point free permutations in $\Sigma_n$. 
\end{theorem}

The number of derangements $d(n)$ is given by the formula $\sum_{i=0}^n(-1)^i\frac{n!}{i!}$, which can be derived using an inclusion-exclusion principle argument. A proof can be found in \cite[Example~2.2.1]{enumerative-book-stanley}. For example, the shelling order of the facets of $\inj[3]$ induced by the lexicographic order $\leq_l$ of the symmetric group $\Sigma_n$ is given by  
\[
	F_1 := 123 \leq F_2 := 132 \leq F_3 := 213 \leq F_4 := 231 \leq F_5 := 312 \leq F_6 := 321,
\]
and $d(3) =2 $. Gluing the simplices along common faces following the order above we get  $\inj[3] \simeq S^2 \vee S^2$. 

\paragraph{Polyhedral products.}
A \emph{diagram of spaces} over a finite poset $\PP$ is a covariant functor $\DD: \PP \to \Top$ from $\PP$ into the category of compact generated topological spaces. Here we consider $\PP$ as a small category with a unique arrow $\sigma \to \tau$ if $\sigma \leq \tau$. This means that for each $\sigma \in \PP$ we associate a topological space $\DD_\sigma$ and to any pair $\sigma \leq \tau$ in $\DD$ we associate a continuous map $d_{\sigma\tau}: \DD_\sigma \to \DD_\tau$ such that $d_{\sigma\sigma} = \id_{\DD_\sigma}$ and $d_{\sigma\mu} = d_{\tau\mu}\circ d_{\sigma\tau}$ for $\sigma\leq\tau\leq\mu$. We call $\DD$ a $\PP$-diagram of spaces.

We shall always consider the pairs $(X_i,A_i)$ to be well-pointed pairs of $CW$-spaces with $A_i \subset X_i$ and base point $x_i \in A_i$ for every $i$ in the vertex set of a simplicial poset $\PP$, unless stated explicitly otherwise. Note that since we work with $CW$-spaces, the pairs $(X_i,A_i)$ are $NDR$-pairs \cite[Proposition~A4]{Hatcher}. Moreover, the inclusion of a base point is a closed cofibration \cite[Corollary~1.3.7]{fritsch_piccinini_1990}.
Therefore, any pointed $CW$-space is well-pointed \cite[Definiton~1.1.5]{cubical-homotopy}.  

\begin{definition}\label{def-poly-space}
	Let $\PP$ be a finite simplicial poset and $\textbf{(X,A)} := \{(X_i,A_i) | i \in V_\PP\}$ a collection of pairs of pointed $CW$-spaces with $A_i \subset X_i$ and base point $x_i \in A_i$ for all $i \in V_\PP$.  The polyhedral product functor determined by $\PP$ and $\textbf{(X,A)}$ is defined as the $\PP$-diagram
		\begin{align*}
			\ZZ: & \PP \to \Top \\
			& \sigma \mapsto \ZZ_\sigma\textbf{(X,A)}
		\end{align*}
		where $\ZZ_\sigma\textbf{(X,A)} := \prod_{i \in V_\PP(\sigma)}X_i \times \prod_{i \not\in V_\PP(\sigma)}A_i$. Whenever $\tau \leq \sigma$ , $\ZZ(\tau \to \sigma)$ maps it to an inclusion $\ZZ_\tau\textbf{(X,A)} \subseteq \ZZ_\sigma\textbf{(X,A)}$.		
		The polyhedral product $\ZZ_{\PP}\textbf{(X,A)}$ is defined as 
		\[
		\ZZ_{\PP}\textbf{(X,A)} := \hocolim_{\sigma \in \PP} \ZZ_\sigma\textbf{(X,A)} 
		\]
	
	The definition is functorial with respect to inclusion maps of $\PP$ and maps between pair of spaces.
\end{definition}

We shall often write $\ZZ_\sigma\textbf{(X,A)}$ as $\ZZ_\sigma$ when $\textbf{(X,A)}$ is clear from context. When $\textbf{(X,A)} := \{(\C P^\infty,*) | i \in V_\PP\}$, the resulting space is called the Davis-Januszkiewicz space, denoted by $DJ_\PP$ or $\ZZ_\PP(\C P^\infty,*)$. Now when $\textbf{(X,A)} := \{(D^2,S^1) | i \in V_\PP\}$, the polyhedral product space is called the moment-angle complex, denoted by $\ZZ_\PP(D^2,S^1)$ or simply by $\ZZ_\PP$.

\paragraph{Colimits and homotopy colimits.} Polyhedral products were first defined as colimits \cite{BBCG,lue-panov}, whereas we defined them as homotopy colimits, following \cite{kishimoto2019polyhedral}. This replacement is useful, as the homotopy colimits have more useful homotopy properties than colimits. Moreover, homotopy colimits have proven to be powerful tools for studying combinatorial problems \cite{homo-type-diagram,hocolim-comb,poset-fiber-thm}. Of course colimits and homotopy colimits are not the same in general, but do coincide up to homotopy in our case as proved in \cite{kishimoto2019polyhedral}. We briefly recall these results for the sake of completeness. We start with the so called Projection Lemma. 

\begin{lemma}\label{proj-lemma}\cite[Proposition~3.1]{hocolim-comb} Let $\DD$ be a $\PP$-diagram of spaces such that $d_{\sigma\tau}$ is a closed cofibration for all $\sigma \leq \tau$ in $\PP$. Then 
\[
\hocolim_\PP \DD \simeq \colim_\PP \DD.
\]
\end{lemma}
The following definition was introduced in \cite[Definition~1.2]{homo-type-diagram}. 
\begin{definition}
	Let $X$ be a topological space and $\mathcal{A} = \{A_1, \ldots, A_r\}$ be a collection of subspaces of $X$. The collection $\mathcal{A}$ is said to be an arrangement, if the following conditions hold:
	\begin{enumerate}
		\item For any $A,B \in \mathcal{A}$, $A\cap B$ is a union of elements in $\mathcal{A}$, and 
		\item for any $A,B \in \mathcal{A}$ such that $A \subset B$, the inclusion $A \to B$ is a cofibration. 
	\end{enumerate}
\end{definition}
The elements of an arrangement $\mathcal{A}$ admit a partial order given by inclusion. Let $I_\mathcal{A}$ be the inclusion poset $(\mathcal{A}, \subseteq)$. We define an $I_\mathcal{A}$-diagram of spaces $\DD(\mathcal{A})$, called \emph{subspace diagram} of $\mathcal{A}$, by setting $\DD_\sigma$ to be the subspace corresponding to $\sigma \in I_\mathcal{A}$, and for each pair $\sigma,\tau \in I_\mathcal{A}$, with $\sigma\leq\tau$, letting $d_{\sigma\tau}$ denote the corresponding inclusion map. Since the intersection of any pair of subspaces in $\mathcal{A}$ is a union of subspaces in $\mathcal{A}$, it follows that $\colim_I \DD(\mathcal{A})$ is homeomorphic to $\bigcup_{A_i \in \mathcal{A}} A_i$. Hence, we get: 

\begin{corollary}\cite[Corollary~2.4]{poset-fiber-thm}
    Let $\mathcal{A}$ be an arrangement of subspaces. Then 
    \[
    \hocolim_{I_\mathcal{A}}\DD(\mathcal{A}) \simeq \bigcup_{A_i \in \mathcal{A}} A_i.
    \]
\end{corollary}

In \cite[Lemma~5.5, Proposition~5.6]{kishimoto2019polyhedral}, Kishimoto and Levi showed that when $\PP$ is a simplicial poset, the collection of spaces $\{\ZZ_\sigma\}_{\sigma \in \PP}$ forms an arrangement with the associated inclusion poset $I_\ZZ$ is isomorphic to $\PP$. Moreover, they proved that the hypotheses of the Projection Lemma \ref{proj-lemma} are satisfied, which implies: 
\begin{equation}\label{union-equiv-levi}
	\ZZ_\PP\textbf{(X,A)} \simeq \hocolim_{I_\ZZ} \DD(\ZZ_\sigma) \simeq \colim_{I_\ZZ} \DD(\ZZ_\sigma) \simeq \bigcup_{\sigma \in \PP} Z_\sigma,
\end{equation}
where $\DD(\ZZ_\sigma)$ is the subspace diagram of $\{\ZZ_\sigma\}_{\sigma \in \PP}$.

Another useful property of homotopy colimits is its relation with homotopy fibres. The following result is due to Puppe.

\begin{theorem}\cite[Appendix~HL]{dror-book}\label{puppe-theo}
Let $\{F_i \to E_i \to B\}_{i \in \II}$ be a $\II$-diagram of homotopy fibrations over a fixed space $B$. Then 
\[\hocolim_\II F_i \to \hocolim_\II E_i \to B\]
is a homotopy fibration. 
\end{theorem}

An immediate application of Puppe's theorem recovers the well-known homotopy fibration \cite[Theorem~4.3.2]{toric-comb-alg-buchstaberpanov}
\begin{align}
\ZZ_K \to DJ_K \to BT^n, \label{classical-homotopy-fibration}
\end{align}
where $K$ is the face poset of simplicial complex on $[n]$ vertices and $BT^n = (\C P^\infty)^n$. More precisely, since any simplicial complex on $[n]$ injects into the simplex $\Delta^{n-1}$ and $D^2 \simeq *$, we have 
\[
\{Z_\sigma(D^2,S^1) \to Z_\sigma(\C P^\infty,*) \to (\C P^\infty)^n\}_{\sigma \in K},
\]
is a homotopy fibration for each $\sigma \in K $and the result follows. More generally, there is a homotopy fibration 
\begin{align}
    \ZZ_\PP \to DJ_\PP \to BT^n \label{folding-homotopy-fibration}
\end{align}
where $\PP$ is a simplicial poset on $[n]$ vertices. This is contructed as follows. For any simplicial poset $\PP$ there is the associated simplicial complex $K_\PP$ on the same vertex set $V_\PP$, whose simplices are the sets $V_\PP(\sigma)$, $\sigma \in \PP$. This is given by a map of simplicial posets 
\[
f: \PP \to K_\PP, \sigma \mapsto V_\PP(\sigma)
\]
called the folding map \cite[Construction~2.8.3]{buchstaber2014toric}. The composition of the folding map with the inclusion map of $K_\PP$ into $\Delta^{n-1}$ induces a map between polyhedral products and the homotopy fibration follows from Theorem \ref{puppe-theo}.

\paragraph{Cohomology of moment angle complexes.} L\"u and Panov in \cite{lue-panov} studied the cohomology of $\ZZ_\PP$ moment-angle complexes associated to simplicial posets, and we recall their results below, starting with the necessary definitions and notation. 

Let $\Z[v_\sigma: \sigma \in \PP]$ be the graded polynomial ring with one generator $v_\sigma$ of degree $\deg v_\sigma = 2 |\sigma|$ for every $\sigma \in \PP$. The Stanley-Reisner ring (over $\Z$) of a simplicial poset $\PP$ is the quotient 
\[
\Z[\PP] = \Z[v_\sigma: \sigma \in \PP] / I_\PP,
\]
where $I_\PP$ denotes the ideal generated by the element $v_{\emptyset} - 1$ and the elements $v_\sigma v_\tau$ whenever $\sigma$ and $\tau$ have no common upper bound in $\PP$, together with the relations $v_\sigma v_\tau - (v_{\sigma \wedge \tau}) \cdot \left(\sum_{\eta \in [\sigma \vee \tau]} v_\eta\right)$ otherwise, where $\sigma \wedge \tau$ and $[\sigma \vee \tau]$ denote, respectively, their greatest common lower bound (meet) and the set of their least common upper bounds (joins). Since $\PP$ is a simplicial poset, $\sigma \wedge \tau$ consists of a single element whenever $[\sigma \vee \tau]$ is non-empty. This definition extends the classical one for simplicial complexes \cite[Definition~3.3]{STANLEY-fhvectors}.

There is a $(\Z \oplus \Z^n)$-graded decomposition of the $Tor$-algebra of the Stanley-Reisner ring $\Z[\PP]$. The grading goes as follows. Firstly, consider the polynomial algebra $\Z[v_1, \ldots, v_n]$ on $n$ generators of degree $2$ corresponding to the vertex set of $\PP$ and denote it by $\Z[n]$. A $\Z^n$-grading in $\Z[n]$ is defined by setting $\deg(v_1^{i_1}\ldots v_n^{i_n}) = (2i_1,\ldots, 2i_n)$. As the Stanley-Reisner ring $\Z[\PP]$ is a $\Z[n]$-algebra via the map $\Z[v_1,\ldots,v_n] \to \Z[\PP]$, which maps each $v_i$ identically, it inherits this $\Z^n$-grading. Recall that $\Z$ has a $\Z[n]$-module structure given by the augmentation map sending each $v_i$ to zero, and the so-called Kozul resolution is a standard way of creating a free resolution of $\Z$. We follow the convention used in \cite{lue-panov}, which numbers the terms in a resolution by non-positive integers. Then, a free (or projective) resolution of a $\Z[n]$-module $M$ is a exact sequence of $\Z[n]$-modules  
\[
\ldots \to P^{-2} \xrightarrow{\delta} P^{-1} \xrightarrow{\delta} P^0 \xrightarrow{\epsilon} M \to 0
\] 
in which $P^{-i}$ is a free (or projective) $\Z[n]$-module for every $i$. Therefore, if we resolve the $\Z[n]$-module $\Z$ with the Kozul resolution and then tensor it with $\otimes_{\Z[n]} \Z[\PP]$, we obtain a $(\Z \oplus \Z^n)$-graded $Tor$-algebra of $\Z[\PP]$ as \[Tor_{\Z[n]}(\Z[\PP],\Z) = \bigoplus_{i \geq 0, a \in \Z^n}Tor^{-i,2a}_{\Z[n]}(\Z[\PP],\Z).\]
For more details, please consult \cite{lue-panov} or \cite[Chapter~4.10,Appendix~A.2]{buchstaber2014toric}.

In \cite{lue-panov}, L\"u and Panov proved the following results about the additive structure of the cohomology of $\ZZ_\PP$.

\begin{theorem}\cite[Theorem~3.5]{lue-panov} \label{theo1-lue-panov}
    There are group isomorphisms
    \[
    H^p(\ZZ_\PP,\Z) \simeq \bigoplus_{-i + 2|a| = p} Tor^{-i,2a}_{\Z[n]}(\Z[\PP],\Z)
    \]
    in each degree $p$. Here $|a| = j_1 + \ldots + j_n$ for $a = (j_1, \ldots, j_n)$. 
\end{theorem}

 \begin{corollary}
     \cite[Corollary~3.9,3.10]{lue-panov} \label{cor1-lue-panov}
    The groups $Tor^{-i,2a}_{\Z[n]}(\Z[\PP],\Z)$ vanish for $a \not\in \{0,1\}^n$ and for every $a \in \{0,1\}^n$ there is an isomorphism 
    \[
    Tor^{-i,2a}_{\Z[n]}(\Z[\PP],\Z) \simeq \tilde{H}^{|a| - i - 1}(|\PP_a|).
    \]
\end{corollary}

Combining Theorem \ref{theo1-lue-panov} and Corollary \ref{cor1-lue-panov} with the substitution $-i = p - 2|a|$, we can reformulate the additive structure of the cohomology of $\ZZ_\PP$ as: 
 \begin{theorem}\label{lue-panov}
    For a simplicial poset $\PP$ on $[n]$, there are the following isomorphism of modules \[ H^0(\ZZ_\PP) \simeq \tilde{H}^{-1}(|P_\emptyset|) = \Z \text{ and } H^p(\ZZ_\PP) \simeq \bigoplus_{a \subset [n]} \tilde{H}^{p - |a| - 1}(|\PP_a|)\]
 \end{theorem}

\section{Homotopy type of $\ZZ_{\Gamma_n}$}\label{sec-homoto-type}

In this section we compute the homotopy type of $\ZZ_{\Gamma_n}$ and establish auxiliary results along the way. We start by determining the additive structure of its cohomology.

\begin{proposition}\label{add-co-inj}
The cohomology  of $\ZZ_{\Gamma_n}$ is isomorphic (as $\Z$-module) to the cohomology of wedge of $h_k$ spheres of dimension $2k$ for each $1 < k \leq n$, where $h_k$ is the $k$th element of the $h$-vector of $\Gamma_n$.
\end{proposition}
\begin{proof}
Note that there are $n \choose k$  copies of $\Gamma_k$ contained in $\Gamma_n$, for $1 \leq k \leq n$, and the $i$-th component of the $f$-vector of $\Gamma_n$ is given by $f_i = {n \choose i+1}(i+1)!$. Recall that the number of derangements in a $k$-element set is $d(k) = \sum_{i=0}^{k} \frac{(-1)^{i}}{i!} k!$.

A calculation shows that $h_k = d(k){n \choose k}$, where $h_k$ is the $k$-th component of $\textbf{h}(\Gamma_n)$ for every $k \leq n$. Indeed, we have 
    \begin{align}
        h_k & := \sum_{i=0}^{k} (-1)^{k-i} {n-i \choose k-i}f_{i-1} \\
        & = \sum_{i=0}^{k} (-1)^{k-i} {n-i \choose k-i}{n \choose i} i! \\
        & = \sum_{i=0}^{k} (-1)^{k-i} \frac{(n-i)!}{(k-i)!(n-k)!}\frac{n!}{(n-i)! i!} i! \\
        & = \left(\sum_{i=0}^{k} (-1)^{k-i} \frac{1}{(k-i)!}\right) \frac{n!}{(n-k)!} \label{eq1}
    \end{align}
On the other hand, we have that 
    \begin{align}
        d(k){n \choose k} & = \sum_{i=0}^{k} \frac{(-1)^{i}}{i!} k! \frac{n!}{(n-k)!k!} \\
        & = \left(\sum_{i=0}^{k} \frac{(-1)^{i}}{i!} \right) \frac{n!}{(n-k)!} \label{eq2}
    \end{align}
Hence, the expressions \eqref{eq1} and \eqref{eq2} are the same  after an index substitution. From Theorem \ref{lue-panov}, we have that $H^*(\ZZ_{\Gamma_n}) \simeq \bigoplus_{J \subset [n]} \tilde{H}^*(|\Gamma_{n_J}|)$. Therefore, 
\[
H^0(\ZZ_{\Gamma_n}) \simeq \tilde{H}^{-1}(|\Gamma_{n_\emptyset}|) \simeq \Z.
\] 
By Theorem \ref{theo-shell-inj}, $\Gamma_k \simeq \vee_{d(k)}S^{k-1}$. For any $k$, $1 < k \leq n$, there are exactly ${n \choose k}$ index sets of size $k$ in $[n]$. Therefore,
    \begin{align}
        H^{2k}(\ZZ_{\Gamma_n}) & \simeq \bigoplus_{k= |a| \subset [n]} \tilde{H}^{k-1}(|\Gamma_{n_a}|) \\ 
        & \simeq \bigoplus_{i = 1}^{{n \choose k}} \tilde{H}^{k-1}(\inj[k])\\
        & \simeq \bigoplus_{i = 1}^{{n \choose k}} \tilde{H}^{k-1}(\vee_{d(k)} S^{k-1})\\
        & \simeq \bigoplus_{i = 1}^{{n \choose k}} \left(\bigoplus_{j=1}^{d(k)}\tilde{H}^{k-1}(S^{k-1}) \right) \simeq \bigoplus_{i = 1}^{h_k} \tilde{H}^{k-1}(S^{k-1})  \label{fin-id}
    \end{align}
and all the rest of the cohomology groups are zero. The isomorphism \eqref{fin-id} follows from the equality $h_k = d(k){n \choose k}$. 
\end{proof}
	
\begin{corollary}
      $H^*(\ZZ_{\Gamma_3})\simeq H^*(S^4) \oplus H^*(S^4) \oplus H^*(S^4) \oplus H^*(S^6) \oplus H^*(S^6)$, where the isomorphism is a ring isomorphism. 
\end{corollary}
\begin{proof}
    The first step is to calculate the $h$-vector of $\Gamma_3$. There are ${3 \choose k+1}(k+1)!$ simplicial $k$-cells in $\Gamma_3$ and none of higher dimension, hence $\textbf{f}(\Gamma_3) = (3,6,6)$. From formula in Definition \ref{dfn-f-h-vector} of the $h$-vector, we have 
    $\textbf{h}(\Gamma_3) = (1,0,3,2)$. Indeed, 
    \begin{align*}
     h_ 0 & = (-1)^{0} {3 \choose 0}f_{-1} = 1. \\ 
     h_1 & = \sum_{i=0}^{1} (-1)^{1-i} {3-i \choose 1-i}f_{i-1} = (-1) {3 \choose 1}f_{-1} + (1)^0 {2 \choose 0}f_{0} = -3 + 3 = 0. \\ 
     h_2 & = \sum_{i=0}^{2} (-1)^{2-i} {3-i \choose 2-i}f_{i-1} = (-1)^2 {3 \choose 2}f_{-1} + (-1) {2 \choose 1}f_{0} + (-1)^0 {1 \choose 0}f_{1} = 3 - 6 + 6 = 3. \\ 
     h_3 & = \sum_{i=0}^{3} (-1)^{3-i} {3-i \choose 3-i}f_{i-1} = (-1)^3 {3 \choose 3}f_{-1} + (-1)^2 {2 \choose 2}f_{0} + (-1)^1 {1 \choose 1}f_{1} + (-1)^0 {0 \choose 0}f_{2} = \\ &  = -1 + 3 - 6 + 6 = 2.
    \end{align*}
    
     Thus, the desired result follows from Proposition \ref{add-co-inj} and due to dimension argument, as there is no $8$-th or higher non-trivial cohomology that could imply a non-trival cup product.
\end{proof}

It is not immediately clear that all cup-products vanish for higher $n$. Before proceeding, we introduce the additional notation required for the next steps.

Suppose that we start with a space $Y$ and another space $X$ that we wish to attach to $Y$ by identifying the points in a subspace $A \subset X$ with points of $Y$. In order to so, we need a map $f: A \to Y$, for then we can form a quotient space of $X \sqcup_f Y$ by identifying each $a \in A$ with its image $f(a) \in Y$. We say that $Y$ is attached to $X$ along $A$ via $f$. The map $f$ is called the attaching map and the resulting space can be described as the pushout of the maps $i: A \to X$ and $f:A \to Y$. The following is a really well-known and useful property of a pair of spaces $(X,A)$ (in particular, pairs of $CW$-spaces) that has the homotopy extension property (HEP) (\cite[Chapter~0]{Hatcher}):

\begin{proposition}\label{hep-conse}\cite[Proposition~0.18]{Hatcher}
	Suppose $(X,A)$ has the HEP, then the homotopy type of $Y\cup_f X$ depends only on the homotopy class of the attaching map $f: A \to Y$ for any space $Y$ and $i:A \to X$ inclusion of subspace $A \subset X$. 
\end{proposition}

Pushouts of ordered simplicial complexes induce homotopy pushouts of polyhedral products.  This is known to be true for simplicial complexes \cite[Proposition~3.8]{toric-homotopy} and we extended it to ordered simplicial complexes. Suppose there is a pushout of ordered simplicial complexes
\[
\begin{tikzcd}
	I \arrow[r, "i"] \arrow[d] & K_2 \arrow[d] \\
	K_1 \arrow[r]  & K 
\end{tikzcd}
\]
where $K := K_1 \cup_I K_2$ is an ordered simplicial complex on $[n]$ vertices. In order to compare the polyhedral products over $L,K_1,K_2,K$ we should really consider each of the four simplicial complexes on the same vertex set $[m]$, as the polyhedral product construction depends on the vertex set of the underlying poset category. To handle this situation, we need the concept of \emph{ghost vertices}. A vertex $i$ is a ghost vertex of an ordered simplicial complex $K$ on $[n]$ vertices if $\{i\}\not\in K$ but $i \in [n]$. This notion allows us to work with multiple ordered simplicial complexes having different numbers of vertices by embedding them all into the same vertex set $[n]$, where we augment each complex by adding its missing vertices as ghost vertices. Denote by $\bar{I},\bar K_1$ and $\bar K_2$ the face posets associated to $L,K_1$ and $K_2$, viewed as ordered simplicial complexes on the vertex set $[n]$ after adding ghost vertices.

\begin{theorem}\cite[Theorem~8.6.1]{cubical-homotopy}\label{cofibration-homotopy-colimit}
    Given a functor $F: \mathcal{I} \to \Top$, where $\Top$ represents the category of compactly generated spaces. The inclusion $\mathcal{J} \to \mathcal{I}$ of a subcategory $\mathcal{J}$ induces a cofibration 
    \[\hocolim_\mathcal{J} F \to \hocolim_\mathcal{I} F.\]
\end{theorem}

\begin{proposition}\cite[Proposition~3.6.17]{cubical-homotopy}\label{hompushout-pushout}
    Consider the diagram 
    \[
    X \xleftarrow{f} W \xrightarrow{g} Y. 
    \]
    If $f$ or $g$ is a cofibration, the natural map $\hocolim( X \xleftarrow{f} W \xrightarrow{g} Y) \to \colim( X \xleftarrow{f} W \xrightarrow{g} Y)$
    is a homotopy equivalence.  
\end{proposition}

\begin{proposition}\label{induce-pushout}
	Let $K$ be a finite ordered simplicial complex on the vertex set $[m]$. Suppose that there is a pushout of ordered simplicial complexes 
	\[
	\begin{tikzcd}
		I \arrow[r] \arrow[d] & K_2 \arrow[d] \\
		K_1 \arrow[r]  & K. 
	\end{tikzcd}
	\]
	Denote by $K, \bar I, \bar K_1, \bar K_2$ the associated face posets on the same vertex set of $K$. Let $\textbf{(X,A)} = \{(X_i,A_i)\}_{i=1}^m$ be a collection of pairs of pointed $CW$-complexes and consider the polyhedral products spaces $\ZZ_{\bar I}\textbf{(X,A)}$, $\ZZ_{\bar K_2}\textbf{(X,A)}$, $\ZZ_{\bar K_1}\textbf{(X,A)}$ and $\ZZ_{K}\textbf{(X,A)}$. Then, the diagram below
	\[
	\begin{tikzcd}
		\ZZ_{\bar I}\textbf{(X,A)} \arrow[r] \arrow[d] & \ZZ_{\bar K_2}\textbf{(X,A)} \arrow[d] \\
		\ZZ_{\bar K_1}\textbf{(X,A)} \arrow[r]  & \ZZ_{K}\textbf{(X,A)},
	\end{tikzcd}
	\] 
	is a homotopy pushout of polyhedral products. By the homotopy equivalence in \ref{union-equiv-levi}, namely $\ZZ_{K}\textbf{(X,A)} \simeq \bigcup_{\sigma \in K} Z_\sigma$, we conclude that $\ZZ_{K}\textbf{(X,A)} \simeq \ZZ_{{\bar K_1}}\textbf{(X,A)}\cup_{\ZZ_{\bar I}\textbf{(X,A)}} \ZZ_{\bar K_2}\textbf{(X,A)}$.  
\end{proposition}
\begin{proof}[Proof of Proposition \ref{induce-pushout}]
	 The inclusions $I \to K_1$ and $I \to K_2$ of ordered simplicial complexes induce the inclusions of polyhedral products $\ZZ_{\bar I}\textbf{(X,A)} \to \ZZ_{\bar K_2}\textbf{(X,A)}$ and $\ZZ_{\bar I}\textbf{(X,A)} \to \ZZ_{\bar K_2}\textbf{(X,A)}$, where the polyhedral products are considered as homotopy colimits. These inclusions are cofibrations by Theorem \ref{cofibration-homotopy-colimit}. Therefore, the canonical projection map  
	\[\begin{tikzcd}
		{\hocolim(\ZZ_{\bar K_1}\textbf{(X,A)} \leftarrow  \ZZ_{\bar I}\textbf{(X,A)} \rightarrow \ZZ_{\bar K_2}\textbf{(X,A)})} \\
		\\
		{\colim(\ZZ_{\bar K_1}\textbf{(X,A)} \leftarrow  \ZZ_{\bar I}\textbf{(X,A)} \rightarrow \ZZ_{\bar K_2}\textbf{(X,A)})}
		\arrow[from=1-1, to=3-1]
	\end{tikzcd}\]
	is a homotopy equivalence by Proposition \ref{hompushout-pushout}. Consequently, the pushout of the diagram 
	\[
	\begin{tikzcd}
		\ZZ_{\bar I}\textbf{(X,A)} \arrow[r] \arrow[d] & \ZZ_{\bar K_2}\textbf{(X,A)} \\
		\ZZ_{\bar K_1}\textbf{(X,A)}
	\end{tikzcd}
	\] 
    can be considered a homotopy pushout (up to homotopy). The rest of the proof follows using the arguments presented in \cite[Proposition~3.8]{toric-homotopy}. Since $K = K_1 \cup_I K_2$ and $K$ is a finite ordered simplicial complex, the simplices in $K$ can be divided into three finite collections: 
	\begin{itemize}
		\item the set $B$ containing the simplices of $I$;
		\item the set $C$ containing the simplices of $K_1$ that are not simplices of $I$;
		\item the set $D$ containing the simplices of $K_2$ that are not simplices of $I$.
	\end{itemize}
	We have 
	\begin{itemize}
		\item $I := \bigcup_{\sigma \in B} \sigma$;
		\item $K_1 := (\bigcup_{\sigma' \in C} \sigma') \cup (\bigcup_{\sigma \in B} \sigma)$;
		\item $K_2 := (\bigcup_{\sigma'' \in D} \sigma'') \cup (\bigcup_{\sigma \in B} \sigma)$;
		\item $K := (\bigcup_{\sigma' \in C} \sigma') \cup (\bigcup_{\sigma'' \in D} \sigma'') \cup (\bigcup_{\sigma \in B} \sigma)$.
	\end{itemize}
	From \cite[Proposition~5.6]{kishimoto2019polyhedral} we have the equivalence described in equation \ref{union-equiv-levi}, which implies that the polyhedral product $\ZZ_\PP\textbf{(X,A)}$ is homotopy equivalent to $\bigcup_{\sigma \in \PP} Z_\sigma$ for any face poset $\PP$ of an ordered simplicial complex on the vertex set $[m]$.  Hence, in our case, we have 
	\begin{itemize}
		\item $\ZZ_{\bar I}\textbf{(X,A)} \simeq \bigcup_{\sigma \in B} Z_\sigma$,
		\item $\ZZ_{\bar K_1}\textbf{(X,A)} \simeq (\bigcup_{\sigma' \in C} Z_{\sigma'}) \cup (\bigcup_{\sigma \in B} Z_\sigma)$,
		\item $\ZZ_{\bar K_2}\textbf{(X,A)} \simeq (\bigcup_{\sigma'' \in D} Z_{\sigma''}) \cup (\bigcup_{\sigma \in B} Z_\sigma)$ and 
		\item $\ZZ_K\textbf{(X,A)} \simeq (\bigcup_{\sigma' \in C} Z_{\sigma'}) \cup (\bigcup_{\sigma \in B} Z_\sigma) \cup (\bigcup_{\sigma'' \in D} Z_{\sigma''})$.
	\end{itemize}
	In particular, since $\ZZ_{\bar I}\textbf{(X,A)} \simeq \ZZ_{\bar K_1}\textbf{(X,A)} \cap \ZZ_{\bar K_2}\textbf{(X,A)}$, we have that 
	\[\ZZ_{K}\textbf{(X,A)} \simeq \ZZ_{\bar K_1}\textbf{(X,A)} \cup_{\ZZ_{\bar I}\textbf{(X,A)}} \ZZ_{\bar K_2}\textbf{(X,A)},\] which implies the existence of the desired homotopy pushout. 
\end{proof}

\begin{proposition}\label{prop-pute-boundary-momentangle}
Let $I$ be a pure (i.e., all facets have same dimension) simplicial complex which is a codimension $1$ subcomplex of the boundary of the $(n-1)$-dimensional simplex $\Delta^{n-1}$. Assume further that $I$ contains all the $n$ vertices of $\Delta^{n-1}$. Then, $\ZZ_I \simeq S^{2|I|-1}$, where $|I|$ is the number of facets in $I$.
\end{proposition}
\begin{proof}
	The set of facets of $I$ is cofinal in the face poset of $I$, hence \[\ZZ_I := \ZZ_{I}(D^2,S^1)= \bigcup_{\sigma \in I} \ZZ_\sigma(D^2,S^1) = \bigcup_{\sigma \in I} (\prod_{i \in V(\sigma)} D^2 \times \prod_{i \not \in V(\sigma)} S^1),\] where the union is taken over the facets of $I$ \cite[Lemma~2.1]{shifted-GRBIC20136}. By hypothesis, the vertex set of $I$ is $[n] = \{1, \ldots, n\}$ and the facets are of codimension $1$, that is, they contain all the vertices of the vertex set but one.
	Therefore, each $\ZZ_\sigma(D^2,S^1)$ is equal to $(\prod_{i = 1}^{n-1} D^2)\times S^1$ for every facet $\sigma$ of $I$, up to reordering the components. The boundary of the cartesian product of closed topological spaces distributes over the product, that is, $\partial (X_1 \times \ldots \times X_i \times \ldots \times X_n) =  (\partial X_1 \times \ldots \times X_i \times \ldots \times X_n) \cup (X_1 \times \ldots \times \partial X_i \times \ldots \times X_n) \cup (X_1 \times \ldots \times X_i \times \ldots \times \partial X_n)$. Thus, there is a homotopy equivalence $\ZZ_I \simeq S^{2|I|-1}$ obtained by contracting the extras $D^2$'s in the product corresponding to each $\ZZ_\sigma(D^2,S^1)$.
\end{proof}

\begin{proposition}\label{push-glue}
	Let $K_1,K_2$ be two $(d-1)$-dimensional simplices $\Delta^{d-1}$ on same vertex set $[d]$ and $I$ be a pure codimension $1$ subcomplex formed by a non-empty set of facets of the boundary of $K_1$ (and $K_2$). Assume further that $I$ contains all the vertices $[d]$ of $\Delta^{d-1}$. Let $K = K_1 \cup_I K_2$ be the ordered simplicial complex obtained as pushout of $K_1$ attached to $K_2$ along $i: I \to K_2$ (i.e. attaching $I$ identically to the corresponding facets of $\partial K_2$). Then, $\ZZ_K \simeq S^{2|I|}$, where $|I|$ is the number of facets of $I$. 
\end{proposition}
\begin{proof}
	By hypothesis $K_1,K_2$ and $I$ have no ghost vertices. Therefore, Proposition \ref{induce-pushout} yields the following homotopy pushout of polyhedral products
	\[
	\begin{tikzcd}
		\ZZ_{I}\arrow[r, "\ZZ(i)"] \arrow[d] & \ZZ_{K_2} \simeq * \arrow[d] \\
		\ZZ_{K_1} \simeq * \arrow[r]  & \ZZ_{K},
	\end{tikzcd}
	\]
	where $\ZZ_{K_1} = \ZZ_{K_1} = (D^{2d}) \simeq *$ since both $K_1$ and $K_2$ are copies of a $(d-1)$-simplex $\Delta^{d-1}$. Moreover, Proposition \ref{prop-pute-boundary-momentangle} showed that $\ZZ_{I} \simeq S^{2|I|-1}$, where $|I|$ is the number of facets of $I$.
 	Hence, the homotopy pushout
	\[
	\begin{tikzcd}
		S^{2|I| -1}\arrow[r, "\ZZ(i)"] \arrow[d] &  (D^{2d}) \simeq * \arrow[d] \\
		(D^{2d}) \simeq *  \arrow[r]  & \ZZ_{K},
	\end{tikzcd}
	\]
     implying that $\ZZ_{K} \simeq \Sigma S^{2|I| -1} \simeq S^{2|I|}$.	
\end{proof}

\begin{corollary}\label{coro-wedge}
	Suppose that there is a linear order $K_1, \ldots, K_n$ of $d-1$-dimensional simplices $\Delta^{d-1}$ on the same vertex set $[d]$ such that $(\cup_{i < j} K_i) \cap K_j$ is a pure subcomplex $I_j$ formed by some of the top dimensional simplices of the boundary $K_j$ containing all the vertices $[d]$, for each $j = 1, \ldots, n$. If $K$ is the complex obtained by gluing each $K_i$ via this (shelling) order, then $\ZZ_K$ is homotopic equivalent to a wedge of spheres. 
\end{corollary}
\begin{proof}
	Denote $(\cup_{i < n} K_i)$ by $W_n$ and $W_n \cap K_n$ by $I_n$. Since, for each $j$, $I_j$ is pure and of codimension $1$ containing all the vertices of the simplex, $I_n$ is also a pure codimension 1 subcomplex. Hence, Proposition \ref{prop-pute-boundary-momentangle} implies that $\ZZ_{I_n} \simeq S^{2|I_n|-1}$, where $|I_n|$ represents the number of facets of $I_n$. Thus, $\ZZ_K \simeq \ZZ_{W_n} \vee S^{2|I_n|}$. Indeed, we have that $\ZZ_{K_n} = D^{2d}$ and Proposition \ref{induce-pushout} implies that the diagram
	\[
	\begin{tikzcd}
		S^{2|I_n| -1}\arrow[r, "\ZZ(i)"] \arrow[d] &  (D^{2d}) \simeq (D^{2|I_n|}) \arrow[d] \\
		\ZZ_{W_n} \arrow[r]  & \ZZ_{K}.
	\end{tikzcd}
	\]
	is homotopy pushout. Following the terminology described in the first paragraph of this section, we say that $(D^{2|I_n|})$ is attached to $\ZZ_{W_n}$ via the map $\ZZ(i): \ZZ_{I} \to (D^{2|I_n|})$. Moreover, the attaching map $\ZZ(i): \ZZ_{I} \to (D^{2|I_n|})$ is null-homotopic, because its target is contractible. This implies that $\ZZ_{K} \simeq \ZZ_{W_n} \vee S^{2|I_n|}$, since we are attaching $(D^{2|I_n|})$ along its boundary. Iterating this procedure for $j = n-1, \ldots, 2$ completes the proof.
\end{proof}

We can now completely determine the homotopy type of $\ZZ_{\Gamma_n}$ by combining the preceding results.
\begin{theorem}\label{homo-type-inj-moment-angle}
\[
\ZZ_{\Gamma_n} \simeq \bigvee_k h_k S^{2k} \text{ with } n\geq 2 \text{ and }  2 \leq k \leq n,
\]
where $h_k$ is the $k$th element of the $h$-vector of the complex of injective words and $h_k S^{2k}$ denotes a wedge sum of $h_k$ spheres of dimension $2k$.
\end{theorem}
\begin{proof}
The complex of injective words $\Gamma_n$ is pure and shellable (recall Theorem \ref{theo-shell-inj}) whose facets are $(n-1)$-dimensional simplexes $\Delta^{n-1}$. Therefore, Corollary \ref{coro-wedge} implies that it is a wedge of spheres. From the additive structure of its cohomology, given by Proposition \ref{add-co-inj}, the desired result follows. 
\end{proof}

\section{A homotopy fibration of polyhedral products over ordered simplicial complexes}\label{sec-fibration}

From the calculation of the homotopy type of $\ZZ_{\Gamma_n}$, we construct a new homotopy fibration of polyhedral products. The motivation comes from following lemma:

\begin{lemma}\label{injection-ord-inj}
	Let $K$ be the face poset of an ordered simplicial complex and $\Gamma_n$ be the face poset of the complex of injective words, both on $[n]$ vertices. 
	Then, there is an injective map of posets $i: K \to \Gamma_n$.
\end{lemma}
\begin{proof}
Recall that an ordered simplicial complex is nothing but a collection of ordered sets closed under taking subsets. Denote the vertex set of $K$ by  $\{v_i\}_{i=1}^n$. Fix an order on the vertices $\{v_i\}_{i=1}^n$ of $K$ and identify them with those of $\Gamma_n$. 

Let $\sigma$ be any $(k-1)$-simplex of $K$. Write $\sigma = (v_{i_1}, \ldots, v_{i_k})$; this ordered tuple determines a permutation of the natural order on $k$ positive integers. Therefore, the same ordered tuple defines a $(k-1)$-simplex of $\Gamma_n$, since $\Gamma_n$ contains all length $k$ permutations as $(k-1)$-simplices. We therefore obtain an assigment 
\[
f: K \to \Gamma_n, \qquad f(\sigma) = \sigma,
\]
sending each simplex of $K$ to the associated simplex of $\Gamma_n$ and, this defines the desired map. In other words, after an identification of vertices of $K$ with those of $\Gamma_n$, the vertex bijection extends to a map of ordered simplicial complexes.
\end{proof}

Therefore, the injective map of posets $i: K \to \Gamma_n$ of Lemma \ref{injection-ord-inj} induces the following map of polyhedral products:
\[
\ZZ_K\textbf{(X,A)} \to \ZZ_{\Gamma_n}\textbf{(X,A)}.
\]

Fix $\textbf{(X,A)} := \{(\C P^\infty,*) | i \in V_{\Gamma_n}\}$ and let $K$ be an ordered simplicial complex on the same vertex set $[n]$. In this section, we adopt the explicit notation $\ZZ_K(\C P^\infty,*)$ for the Davis-Januszkiewicz space and $\ZZ_K(D^2,S^1)$ for the moment-angle complex in this section, as it makes the application of Puppe's Theorem \ref{puppe-theo} more transparent. We want to understand the homotopy fibre $F_K$ of 
\[
\ZZ_K(\C P^\infty,*)\to \ZZ_{\Gamma_n}(\C P^\infty,*).
\]

We can construct the following commutative diagram 
\[
\begin{tikzcd}
	F_K  \arrow[r] \arrow[d,swap,"\simeq"] & \ZZ_K(D^2,S^1) \arrow[d] \arrow[r] & \ZZ_{\Gamma_n}(D^2,S^1) \arrow[d]\\
    F_K \arrow[d] \arrow[r] & \ZZ_K(\C P^\infty,*) \arrow[d]\arrow[r] &  \ZZ_{\Gamma_n}(\C P^\infty,*) \arrow[d,]\\ 
	* \arrow[r] & (\C P^\infty)^{n} \arrow[r] & (\C P^\infty)^{n},
\end{tikzcd}
\]
where the vertical and horizontal lines are homotopy fibrations. The second and third vertical fibrations are constructed analogously to the homotopy fibration \ref{folding-homotopy-fibration}. To determine the homotopy type of $F_K$, it therefore suffices to the study the homotopy fibration 
\[
F_K \to \ZZ_K(D^2,S^1)\to \ZZ_{\Gamma_n}(D^2,S^1)
\]
up to homotopy. 

For instance, if $K = \Delta^{n-1}$, the diagram above can be rewritten as follows
\[
\begin{tikzcd}
    F_K  \arrow[r] \arrow[d,swap,"\simeq"] & * \arrow[d] \arrow[r] & \ZZ_{\Gamma_n}(D^2,S^1) \arrow[d]\\
    F_K \arrow[d] \arrow[r] & (\C P^\infty)^{n} \arrow[d,swap,"\simeq"]\arrow[r] &  \ZZ_{\Gamma_n}(\C P^\infty,*) \arrow[d  ]\\ 
    * \arrow[r] & (\C P^\infty)^{n} \arrow[r] & (\C P^\infty)^{n}.
\end{tikzcd}
\]
Hence, $F_K \simeq \Omega \ZZ_{\Gamma_n}(D^2,S^1)$. In particular, for $n=2$, Theorem \ref{homo-type-inj-moment-angle} yields 
$\ZZ_{\Gamma_2}(D^2,S^1) \simeq S^4$, giving the homotopy fibration
\[
\Omega S^4 \to (\C P^\infty)^{\times 2} \to \ZZ_{\Gamma_2}(\C P^\infty,*).
\]

In order to understand the situation when $K$ is an arbitrary ordered simplicial complex, we need to recall some well known facts. 

\begin{lemma}\cite[Lemma~2.2.10]{cubical-homotopy}\label{lemma-above-1}
	Let $f = *: X \to Y$ be the constant map, that is, $x \mapsto *$ for all $x$ in $X$. Then, the homotopy fibre $F_f$ is homotopic to $X \times \Omega Y$.
\end{lemma}

\begin{corollary}\label{lemma-above-2}\cite[Corollary~4.9]{Hatcher}
Any map $S^i \to S^n$, where $0 < i < n$ is null homotopic, that is, homotopic to the constant map.
\end{corollary}
\begin{corollary}\label{coro-fibration}
	The homotopy fiber of a map $f: S^k \to S^n$ is homotopy equivalent to $S^k \times \Omega S^n$, where $1 \leq k < n$.
\end{corollary}
\begin{proof}
	It follows immediately from Lemma \ref{lemma-above-1} and Corollary \ref{lemma-above-2}.  
\end{proof}

\begin{theorem}\label{mark-suggestion-new-fibration}
	There is a homotopy fibration sequence
	\[
    \ZZ_K(\Omega(\bigvee_k h_k S^{2k})\times D^2, \Omega(\bigvee_k h_k S^{2k})\times S^1 )	\to \ZZ_K(\C P^\infty,*)\to \ZZ_{\Gamma_n}(\C P^\infty,*),
	\]
	where $K$ and $\Gamma_n$ denote, respectively, the face poset of an ordered simplicial complex and the face poset of the complex of injective words, both on $[n]$ vertices with $n \geq 2$, and $h_k$ is the $k$th entry of the $h$-vector of $\Gamma_n$.
\end{theorem}
\begin{proof}
	Recall the following diagram 
	\[
	\begin{tikzcd}
	F_K  \arrow[r] \arrow[d,swap,"\simeq"] & \ZZ_K(D^2,S^1) \arrow[d] \arrow[r] & \ZZ_{\Gamma_n}(D^2,S^1) \arrow[d]\\
	F_K \arrow[d] \arrow[r] & \ZZ_K(\C P^\infty,*) \arrow[d]\arrow[r] &  \ZZ_{\Gamma_n}(\C P^\infty,*) \arrow[d]\\ 
	* \arrow[r] & (\C P^\infty)^{\times n} \arrow[r] & (\C P^\infty)^{\times n},
	\end{tikzcd}
	\]
	Recall Theorem \ref{homo-type-inj-moment-angle} establishes $\ZZ_{\Gamma_n}(D^2,S^1) \simeq \bigvee_k h_k S^{2k}$, where $h_k$ denotes the $k$th entry of the $h$-vector of $\Gamma_n$.
	
	Fix an element $\sigma \in K$. Lemma \ref{lemma-above-1} and Corollary \ref{lemma-above-2} yield the homotopy fibration sequence
	\begin{equation}\label{htpy-fibration1}
	\Omega(\bigvee_k h_k S^{2k})\times S^1 \to S^1 \to \bigvee_k h_k S^{2k}. 
	\end{equation}
	  There is also a homotopy fibration 
    \begin{equation}\label{htpy-fibration2}
    \Omega(\bigvee_k h_k S^{2k})\times D^2 \to D^2 \to \bigvee_k h_k S^{2k}
    \end{equation}
    and \ref{htpy-fibration1} maps to \ref{htpy-fibration2} via including $S^1$ into $D^2$. By Puppe's theorem \ref{puppe-theo}, this yields a homotopy fibration 
	\[
	\ZZ_\sigma(\Omega(\bigvee_k h_k S^{2k}) \times D^2, \Omega(\bigvee_k h_k S^{2k})\times S^1 ) \to	\ZZ_\sigma(D^2,S^1)\to \bigvee h_k S^{2k}.
	\]
	Since this construction applies to any $\sigma \in K$, we obtain a family of homotopy fibrations indexed by elements of $K$:
	\[
	\{\ZZ_\sigma(\Omega(\bigvee h_k S^{2k})\times D^2, \Omega(\bigvee h_k S^{2k})\times S^1 ) \to	\ZZ_\sigma(D^2,S^1)\to \bigvee h_k S^{2k}\}_{\sigma \in K}.
	\]
	  Applying Puppe's Theorem \ref{puppe-theo} again, this family induces a homotopy fibration 
    \[\ZZ_K(\Omega(\bigvee h_k S^{2k})\times D^2, \Omega(\bigvee h_k S^{2k})\times S^1 ) \to \ZZ_K(D^2,S^1)\to \bigvee h_k S^{2k}
    \]
    which completes the proof.
\end{proof}

For example, let $K$ be a disjoint union of two vertices. We have the following diagram of homotopy fibrations:
		\[
	\begin{tikzcd}
		F_K  \arrow[r] \arrow[d,swap,"\simeq"] & \ZZ_{K}(D^2,S^1) \simeq S^3 \arrow[d] \arrow[r] & \ZZ_{\Gamma_2}(D^2,S^1) \simeq S^4 \arrow[d]\\
		F_K \arrow[d] \arrow[r] &  \ZZ_{K}(\C P^\infty,*) \simeq \C P^\infty \vee \C P^\infty \arrow[d,swap,]\arrow[r] &  \ZZ_{\Gamma_2}(\C P^\infty,*) \arrow[d,]\\ 
		* \arrow[r] & (\C P^\infty)^{\times 2} \arrow[r] & (\C P^\infty)^{\times 2}.
	\end{tikzcd}
	\]
By Theorem \ref{mark-suggestion-new-fibration} and Corollary \ref{coro-fibration}, we conclude that $\ZZ_K(\Omega S^4 \times D^2, \Omega S^4 \times S^1) \simeq \Omega S^4 \times S^3$.

\section*{Acknowledgments} 
 This paper is based in part on a chapter of the PhD thesis of the author \cite{pedro_phd}. The author thanks Ran Levi for his support and insightful discussions throughout this work, and the anonymous referee for helpful comments that improved the exposition of this article.

\printbibliography

\end{document}

%% file: header.tex
\usepackage[margin=25mm]{geometry}
\usepackage{graphicx,tikz}
\usepackage{tikz-cd}
\usetikzlibrary{calc, arrows.meta} 
\usepackage{multirow}
\usepackage{subcaption}
\usepackage[table]{xcolor}
\usepackage{xcolor}
\usepackage{csquotes} 
\usepackage[
  backend=biber,
  style=numeric-comp,
  sorting=nty,
  natbib=true,
  maxbibnames=10,
  uniquelist=false,
  uniquename=false
]{biblatex}
\usepackage{fancyhdr}

\bibliography{polyinj}
 \AtEveryCitekey{\clearlist{location}}
 \AtEveryBibitem{\clearlist{location}}
\makeatletter
\newcommand\footnoteref[1]{\protected@xdef\@thefnmark{\ref{#1}}\@footnotemark}
\makeatother

\usepackage[algo2e,vlined,ruled,linesnumbered,nofillcomment]{algorithm2e} 
\SetKwIF{If}{ElseIf}{Else}{if}{}{else if}{else}{end if}%
\SetKwFor{While}{while}{}{end while}%
\SetKwFor{For}{for}{}{end for}%
\DontPrintSemicolon

\usepackage{listings}
\lstdefinestyle{snippets}{
    backgroundcolor=\color{bg},   
    commentstyle=\color{ggreen},
    keywordstyle=\color{gred},
    numberstyle=\tiny\color{gray},
    stringstyle=\color{gblue},
    identifierstyle=\color{gblue},
    basicstyle=\ttfamily\footnotesize,
    breakatwhitespace=false,         
    breaklines=true,                 
    captionpos=b,                    
    keepspaces=true,                 
    numbers=left,                    
    numbersep=5pt,                  
    showspaces=false,                
    showstringspaces=false,
    showtabs=false,                  
    tabsize=2,
    xleftmargin=3mm,
    xrightmargin=3mm,
    framexleftmargin=3mm,
    framexrightmargin=3mm,
    framextopmargin=0mm,
    framexbottommargin=0mm,
    frame=tlbr,framesep=5pt,framerule=0pt
}
\usepackage{sectsty}
\sectionfont{\fontsize{12}{15}\selectfont\centering}

\usepackage{amsmath,amssymb}
\usepackage{colonequals}

\usepackage{hyperref} 
\usepackage[noabbrev,capitalize]{cleveref}

\usepackage{amsthm,thmtools}
\newtheorem{theorem}{Theorem}[section]
\newtheorem{proposition}[theorem]{Proposition}
\newtheorem{corollary}[theorem]{Corollary}
\newtheorem{lemma}[theorem]{Lemma}

\theoremstyle{definition}
\newtheorem{definition}[theorem]{Definition}

\newtheorem{alphabeticaltheorem}{Theorem}[theorem] 

\newcommand{\inj}{\mathbf{Inj}}

\newcommand{\id}{\text{id}}

\DeclareMathOperator*{\hocolim}{hocolim}
\DeclareMathOperator*{\colim}{colim}
\newcommand{\Top}{\textbf{Top}}

\newcommand{\Z}{\mathbb{Z}}

\newcommand{\C}{\mathbb{C}}

\newcommand{\DD}{\mathcal{D}}

\newcommand{\II}{\mathcal{I}}
\newcommand{\PP}{\mathcal{P}}
\newcommand{\ZZ}{\mathcal{Z}}

\crefname{definition}{Definition}{Definitions}

\AtEndDocument{
  \par
  \medskip
  \noindent
  \begin{tabular}{p{\textwidth}}%
    Institute of Computer Science, Dependable Systems, Kiel University, Kiel, 24118, Germany\\
    \emph{Email address}: \texttt{prdc@informatik.uni-kiel.de}
  \end{tabular}}

